\documentclass[11pt]{amsart}
\usepackage{latexsym}
\usepackage{amssymb}
\usepackage{amsmath}

\usepackage[normalem]{ulem}

\usepackage{hyperref}%

\usepackage{amsfonts}

\usepackage{tikz}

\usetikzlibrary{snakes}
\usetikzlibrary{calc,positioning}
\usetikzlibrary{3d,calc}

\newtheorem{theorem}{Theorem}[section]
\newtheorem{lemma}[theorem]{Lemma}
\newtheorem{proposition}[theorem]{Proposition}
\newtheorem{corollary}[theorem]{Corollary}
\newtheorem{definition}[theorem]{Definition}

\theoremstyle{definition}
\newtheorem{remark}[theorem]{Remark}

\newcommand\sspp{\mathop{\rm span}}
\newcommand\vn{\mathop{\rm VN}}

\newcommand\Bim{\mathop{\rm Bim}}
\newcommand\Sat{\mathop{\rm Sat}}
\newcommand\Ran{\mathop{\rm Ran}}

\def\ga{\alpha}

\newcommand\an{^{-1}}

\newcommand\cb{\mathop{\rm cb}}

\def\eps{\epsilon}
\def\gl{\lambda}
\def\gd{\delta}
\def\gG{\Gamma}

\def\gs{\sigma}

\def\Gs{\Sigma}

\newcommand{\cl}[1]{\mathcal{#1}}
\newcommand{\bb}[1]{\mathbb{#1}}
\newcommand{\du}[2]{\left\langle{#1},{#2} \right\rangle} 

\def\gd{\delta}

\def\ot{\otimes}
\newcommand{\nor}[1]{\left\Vert #1\right\Vert}    
\newcommand{\sca}[1]{\left(#1\right)} %

\newcommand\Ad{\mathop{\rm ad}}

\begin{document}

\title[ Bimodules over $\vn(G)$ and  the  Poisson boundary]{Bimodules over $\vn(G)$, harmonic operators and  the non-commutative Poisson boundary}

\author{M. Anoussis, A. Katavolos and  I. G. Todorov}

\begin{abstract}  Starting with a left ideal $J$ of $L^1(G)$ we consider its annihilator $J^{\perp}$ 
in $L^{\infty}(G)$ and the generated $\vn (G)$-bimodule in $\cl B(L^2(G))$, $\Bim(J^{\perp})$. 
We prove that 
$ \Bim(J^{\perp})=(\Ran J)^{\perp}$ when $G$ is weakly amenable discrete, compact or abelian, where $\Ran J$ is 
a suitable saturation  of $J$ in the trace class.
We define jointly harmonic functions and jointly harmonic operators and 
show that, for these classes of groups, the space of jointly harmonic operators is the 
$\vn(G)$-bimodule generated by the space of jointly harmonic functions. 
Using this, we give a proof of   the following  result of Izumi and Jaworski -- Neufang: 
 the non-commutative Poisson boundary is isomorphic to the crossed product  of the space 
of harmonic functions by $G$.  
\end{abstract}

\address{Department of Mathematics, University of the Aegean,
Samos 83 200, Greece}

\email{mano@aegean.gr}

\address{Department of Mathematics, University of Athens,
Athens 157 84, Greece}

\email{akatavol@math.uoa.gr}

\address{Mathematical Sciences Research Centre, Queen's University Belfast,
Belfast BT7 1NN, United Kingdom}

\email{i.todorov@qub.ac.uk}

\date{}

\maketitle

  \section{introduction}
Let $J$ be an ideal of the Fourier algebra $A(G)$ of a locally compact group $G$. 
There are two `canonical' ways to construct from $J$ an $L^{\infty}(G)$-bimodule of 
$\cl B(L^2(G))$. One way is to consider the annihilator $J^{\perp}$ of $J$ 
within $\vn (G)$ and then take the $L^{\infty}(G)$-bimodule
generated by $J^{\perp}$, denoted by  $\Bim( J^\perp)$. The other way is to
take the saturation of $J$ within the trace class on $L^2(G)$, which we call 
$\Sat J$, and then consider its annihilator. This gives 
a masa bimodule $(\Sat J)^{\perp}$ in $\cl B(L^2(G))$. In \cite{akt},
we proved that these two procedures  yield the same bimodule, that is, 
\[\Bim (J^\perp)=(\Sat J)^\perp. \tag{$\ast$}\]
    
In \cite{nr}, Neufang and Runde introduced the notion of $\gs$-harmonic
operators $\widetilde{\cl H}_\gs$ (where $\gs$ belongs to the space of completely bounded multipliers 
$ M^{\cb}A(G)$  of $A(G)$)
as an extension of the notion of $\gs$-harmonic 
functionals on $A(G)$ defined and studied by Chu and Lau in \cite{cl}. 
One of the main results of \cite{nr} is that, when $\gs$ is positive definite and normalised,
$\widetilde{\cl H}_\gs$ is the von Neumann algebra on $L^2(G)$ generated 
by the algebra $\cl D_G$ of multiplication operators together 
with the space ${\cl H}_\gs$ of harmonic functionals. 
In \cite{akt2}, for a subset  $\Gs\subseteq M^{\cb}A(G)$
we considered the set of jointly harmonic functionals $\cl H_{\Sigma}$
 (resp. operators $\widetilde{\cl H}_{\Sigma}$). 
Using the equality $(*)$,  
we showed that, for any $\Sigma\subseteq M^{\cb}A(G)$, we have $\widetilde{\cl H}_\Gs=\Bim({\cl H}_\Gs)$,
thus obtaining a generalization of the result of Neufang and Runde.

Another concept of harmonicity is introduced and studied by Jaworski and Neufang in \cite{jn}.   
Recall that a function $\phi\in L^\infty(G)$ is said to be harmonic with respect to  a probability measure
$\mu$  on $G$  \cite{who,furst} if it is a fixed point point of the map $P_\mu$ on $L^\infty(G)$ given by
$$ (P_\mu\phi)(s) = \int_G\phi(st)d\mu(t) \, .$$
The space of $\mu$-harmonic functions 
is denoted by $\cl H(\mu)$.
If $G$ is abelian, it follows from the Choquet-Deny theorem that, if the support of $\mu$ generates 
$G$ as a closed subgroup, then
$\cl H(\mu)$ consists of constants. 
In particular, it is a subalgebra of $L^{\infty}(G)$. 
Consider the natural isometric representation $\mu\to\Theta(\mu)$ of the measure algebra 
$M(G)$ on $\cl B(L^2(G))$ introduced by Ghahramani in \cite{gh}.  For $\mu \in M(G)$, 
the map $\Theta(\mu)$ extends the action 
$\phi\to P_\mu(\phi), \, \phi\in L^\infty(G)$.
For a probability measure $\mu$, the harmonic operators $T$ 
are defined in \cite{jn} by the relation $\Theta(\mu)T=T$. 
The collection of all $\mu$-harmonic operators is denoted by $\widetilde{\cl H} (\mu)$. 
The non-commutative Poisson boundary of $\mu$, denoted by $\widetilde{\cl H}_\mu$, is defined  to be the space 
$\widetilde{\cl H}(\mu)$,  equipped with  a certain von Neumann algebra structure \cite{izadv}.
The space $\cl H(\mu) $   
is a von Neumann subalgebra of $\widetilde{\cl H}(\mu)$ denoted by $\cl H_{\mu}$. 
Non-commutative Poisson boundaries were first considered by Izumi for discrete  groups in \cite{izu} 
where he showed that
$\widetilde{\cl H}_\mu$ is the crossed product of $\cl H_\mu$ by $G$ acting by left translations.
Jaworski and Neufang extended this in 
\cite{jn} to locally compact $G$, thus answering a question in \cite{izu}. 
This result was further generalised  in \cite{knr}  for locally compact quantum groups. 
  
When $G$ is abelian, the settings described in the previous two paragraphs 
are connected by the usual Fourier transform. 
(In particular, $\widetilde{\cl H} (\mu)$  is a subalgebra of $\cl B(L^2(G))$ in this case.) 
We discuss this relation in Section 4.

One may ask: What is a dual version of $(*)$? Can it be used to study  the space $\widetilde{\cl H} (\mu)$ of 
harmonic operators?
The present paper focuses on these questions. Instead of an ideal of $A(G)$, we start with a  left ideal $J$  of $L^1(G)$. 
We then   consider its annihilator $J^{\perp}$ in $L^{\infty}(G)$ and the $\vn (G)$-bimodule 
$\Bim(J^{\perp})$ generated by the collection of multiplication operators
$\{M_f: f\in J^\perp\}$ in $\cl B(L^2(G))$. 
We also construct a suitable saturation $\Ran J$ of $J$ within the trace class  $\cl T(G)$ on $L^2(G)$.
When $G$ is abelian, utilising Fourier transform and using $(*)$,
we show (Section \ref{s_ac}) that 
$$(\Ran J)^{\perp} = \Bim(J^{\perp}).$$
The following question then arises: Is this formula true for any locally compact group $G$? 
We show that equality does occur when $G$ is weakly amenable discrete 
(Section \ref{s_dc}) or compact (Section \ref{s_cc}).

Given a set $\Lambda\subseteq M(G)$ (not necessarily consisting of probability measures),
in Section \ref{s_jho} we define the space of {\em jointly $\Lambda$-harmonic} functions $\cl H(\Lambda)$  
to be the set of functions in $L^{\infty}(G)$ which are $\mu$-harmonic 
for every $\mu$ in $\Lambda$, and
we introduce in an analogous fashion the corresponding space of {\em jointly $\Lambda$-harmonic operators} 
$\widetilde{\cl H}(\Lambda)$. 
As a consequence of our previous results, we recover 
 $\widetilde{\cl H}(\Lambda)$, when the group is compact, weakly amenable discrete  or abelian:  
we show that it is the weak-* closed 
$\vn(G)$-bimodule generated by   ${\cl H}(\Lambda)$ in $\cl B(L^2(G))$. 
In the case where $\Lambda$ is a singleton consisting of a probability measure $\mu$, using  this we
give a proof of the above mentioned  result of Izumi and Jaworski -- Neufang: the 
non commutative Poisson boundary  $\widetilde{\cl H}_\mu$ is isomorphic
to the  crossed product of $\cl H_\mu$ by a canonical action of $G$.

\section{Preliminaries}\label{s_prel}

Let $G$ be  a second countable locally compact group equipped with left Haar measure. As usual, the 
corresponding Lebesgue spaces on $G$ are denoted by $L^p(G)$ for $1\le p \le\infty$.
We denote by $\lambda : G\rightarrow \cl B(L^2(G))$, $s\to \lambda_s$ the left
regular representation of the group $G$, {given by $(\gl_sf)(t)=f(s\an t)$}; 
here, $\mathcal{B}(L^2(G))$ denotes the algebra of bounded linear operators on  $L^2(G)$. 
We write
$(\cdot,\cdot)$ for the inner product
and we use $\du{\cdot}{\cdot}$ for the various Banach space dualities, 
in particular for the duality between $L^1(G)$ and $L^{\infty}(G)$. 
For $\phi\in L^{\infty}(G)$, let $M_{\phi}$ be the operator on $L^2(G)$
of multiplication by $\phi$.
We denote by $\cl D_G$ or $\cl D$ the algebra 
$ \{M_{\phi} : \phi\in L^{\infty}(G)\}$. This is a maximal
abelian selfadjoint algebra (for brevity, masa).

The predual $\cl T(G)$ of  $\mathcal{B}(L^2(G)) $ can be identified with the space of all functions  the form
 $h:G\times G\to\bb C$, defined marginally almost everywhere (see for example \cite{akt} ) and given by 
\begin{equation}\label{h}
 h(x,y)=\sum_{i=1}^{\infty} f_i(x )g_i(y),
\end{equation}
where
$\sum\limits_{i=1}^{\infty}\nor{f_i}_2^2<\infty$ and $\sum\limits_{i=1}^{\infty}\nor{g_i}_2^2<\infty$.
The norm on $\cl T(G)$ is given by 
\[
\|h\|_t=\inf\left\{\sum\limits_{i=1}^{\infty}\nor{f_i}_2\nor{g_i}_2\right\} 
\]
where the infimum is taken over all representations (\ref{h}) of $h$.  
The pairing between   $\mathcal{B}(L^2(G)) $ and  $\cl T(G)$ is given by
\[
 \du{T}{h}_t := \sum_{i=1}^{\infty} \sca{Tf_i,\bar g_i}.
\]

The \emph{group von Neumann algebra} of $G$ is the algebra
$$\vn(G) =\overline{\sspp\{\lambda_x : x\in G\}}^{w*},$$ acting on $L^2(G)$. Its predual 
can be identified with the \emph{Fourier algebra} $A(G)$ of $G$ \cite{eym}
which is the (commutative, regular, semi-simple) Banach algebra
consisting of
all complex functions $u$ on $G$ of the form
\begin{equation}\label{ag}
u(x) = (\lambda_x f,g),\ \  x\in G, \mbox{ where } f,g\in L^2(G).
\end{equation}
The pairing between   $\vn(G)$ and  $A(G)$
is given by $\du{\gl_x}{u}_A=u(x)$.
A function $\gs: G\rightarrow \bb{C}$ 
is called a \emph{multiplier} of $A(G)$ if $\sigma u \in A(G)$ for every $u\in A(G)$. 
If $\sigma$ is a multiplier of $A(G)$, the map $m_{\sigma} : A(G) \to A(G)$, given by 
$m_{\sigma}(u) = \sigma u$, is automatically bounded. 
A multiplier $\sigma$ of $A(G)$ is called \emph{completely bounded} \cite{deCH} 
if the dual 
$m_{\sigma}^* : \vn(G)\to \vn(G)$ of $m_{\sigma}$ is completely bounded. 
We write $M^{\cb}A(G)$ for the algebra of all completely bounded multipliers of $A(G)$. 
If $\sigma$ is in $M^{\cb}A(G)$ and $h\in\cl T(G)$, it was shown by 
J. E. Gilbert and M. Bo$\dot{{\rm z}}$ejko-G. Fendler 
in \cite{bf} that $N(\sigma) h$ is in $\cl T(G)$, where $N(\gs)(s,t) = \gs(ts^{-1})$. 

Let $J$ be a closed ideal of  $ A(G)$.   Consider the
norm closed masa bimodule 
$$\Sat J= \overline{\sspp(N(J)\cl T(G))}^{\|\cdot\|_t}$$
of $\cl T(G)$ generated by
$N(J)$.  Denote by $(\Sat J)^\perp$  the annihilator of $\Sat J$ in $\cl B(L^2(G))$.
Let $J^{\perp}$ be the annihilator of $J$ in $\vn(G)$,
and $\Bim(J^{\perp})$ be the weak-* closed masa bimodule generated by $J^\perp$ in $\cl B(L^2(G))$.

The following result was proved in \cite{akt}: 

\begin{theorem}\label{th_satlcg}
Let $J\subseteq A(G)$ be a closed ideal. Then
$(\Sat J)^\perp = \Bim (J^\perp)$.
\end{theorem}

\section{Ideals of $L^1(G)$ and bimodules over $\vn(G)$}\label{s_s}

Throughout this section, we fix a locally compact group $G$. 
Let $\rho : G\rightarrow \cl B(L^2(G))$, $r\rightarrow \rho_r$, be the
right regular representation of $G$ on $L^2(G)$, given by 
$$(\rho_r f)(s) = \Delta(r)^{1/2} f(sr), \ \ \ f \in L^2(G), \ s,r\in G,$$ 
where $\Delta$ denotes the modular function of $G$.

Denote by $\Ad\rho_r$ the map on $\cl B(L^2(G))$ given by
$\Ad\rho_r(T)$ $= \rho_r T \rho_r^*$, $T\in \cl B(L^2(G))$.
Let $M(G)$ be the measure algebra of $G$, that is the 
(convolution) Banach algebra of all bounded, complex Borel measures on $G$.
We identify $L^1(G)$ with the (closed) ideal of $M(G)$ consisting of all measures, 
absolutely continuous with respect to Haar measure.
Define a representation 
$\Theta$ of the  algebra $M(G)$ on $\cl B(L^2(G))$ by
$$\langle\Theta(\mu) (T),h\rangle_t
= \int_{G} \langle\Ad\rho_r(T),h\rangle_t d\mu(r)$$
for every $h\in T(G)$.
This representation was introduced  and studied  by E. St\o{}rmer for abelian groups  
\cite{stor} and by F. Ghahramani \cite{gh} for  locally compact groups.
See \cite{nrs} for more references.

Since $\Ad\rho_r$ and $\Theta(\mu)$ are (bounded) weak-* continuous maps, 
they have (bounded) preduals $\theta_r$ and $\theta(\mu) : \cl T(G)\to \cl T(G)$.
Thus,
$$\theta(\mu)(h) = \int_G\theta_r(h) d\mu(r), \ \ \ h\in \cl T(G).$$
Note that, for $r\in G$, we have \cite[Lemma 4.1]{akt}
\begin{equation}\label{eq_htr}
\theta_r(h) = \Delta(r^{-1})h_{r^{-1}}, \ \ \ h\in \cl T(G).
\end{equation}
Here, $h_r(s,t) = h(sr,tr)$, $s,t,r\in G$.
Therefore,  if $f \in L^1(G)$ then
$$\theta(f)(h) = \int_G \Delta(r^{-1})h_{r^{-1}} f(r) dr, \ \ \ h\in \cl T(G).$$

Let $J\subseteq L^1(G)$ be a closed left ideal; we denote  by $J^{\perp}$ its annihilator in $L^{\infty}(G)$.
Set
$$\Ran J = \overline{\sspp{\{\theta(f)(h) : f\in J, h\in \cl T(G)\}}}^{\nor{\cdot}_t}\subseteq \cl T(G) .$$
Given a subspace $\cl U\subseteq L^{\infty}(G)$, we let
$$\Bim(\cl U) =  \overline{\sspp{\{A M_a B : A,B\in \vn(G), a\in \cl U\}}}^{w^*}\subseteq\cl B(L^2(G))\, ;$$
thus, $\Bim(\cl U)$ is the weak-* closed $\vn(G)$-bimodule generated by the
multiplication operators with symbols coming from $\cl U$.

\medskip
We denote by $(\Ran J)^{\perp}$ the annihilator of $\Ran J$  within $\cl B(L^2(G))$. 
We are interested in the relation between $(\Ran J)^{\perp}$ and $\Bim(J^{\perp})$.
\bigskip

\begin{lemma}\label{l}
The space  $(\Ran J)^{\perp}$ is the intersection of the kernels
of the maps $\{\Theta(f):f\in J\}$. We write this as
\[(\Ran J)^{\perp}=\ker\Theta(J).\]
Consequently, $(\Ran J)^{\perp}$ is a $\vn (G)$-bimodule.
\end{lemma}

\begin{proof}
Since $\Theta(f)$ is a $\vn (G)$-bimodule map for every $f\in J$, the space 
$\ker\Theta(J)$ is a $\vn (G)$-bimodule.
The equality $(\Ran J)^{\perp}=\ker\Theta(J)$ follows directly from the definition.
\end{proof}

\begin{remark}\label{rem28} 
Let $s,t\in G$, $f\in L^1(G)$ and $a\in L^\infty(G)$. Then 
\[
\Theta(f)(\gl_s^*M_a\gl_t)  = \gl_s^*\Theta(f)(M_a)\gl_t =  \gl_s^*\left(\int_G \rho_rM_a\rho_r^*f(r)dr\right)\gl_t 
= \gl_s^*M_g\gl_t,
\] where
\[
g(x) = \int_G a(xr)f(r)dr = \int_G f(x\an z) a(z)dz = \du{a}{\gl_xf}, \ \ x\in G.
\]
\end{remark}

\begin{lemma}\label{easy}
Let $s,t\in G$ and $a\in L^\infty(G)$. Then 
\[ \gl_s^*M_a\gl_t\in (\Ran J)^{\perp} \;\iff\; a\in J^\bot.\]
\end{lemma}
\begin{proof} 
Since $(\Ran J)^{\perp}$ is a $\vn (G)$-bimodule, it suffices to show that $a\in J^\bot$ if and only if
$M_a\in (\Ran J)^{\perp}$.

Suppose $a\in J^\perp$ and $f \in J$. By Remark \ref{rem28}, 
\begin{align*}
\Theta(f)(M_a) = M_g,  \quad\text{where} \quad
g(x) = \du{a}{\gl_xf}, \ x\in G.
\end{align*}
Since $f\in J$ and $J$ is a closed left ideal, 
$\gl_xf \in J$ \cite[2.43]{follharm}, and so $g$ vanishes almost everywhere.
Thus, $\Theta(f)(M_a)=0$ for all $f\in J$ and so  $M_a\in (\Ran J)^\perp$.

Suppose, conversely, that $M_a\in (\Ran J)^{\perp}$. 
Then for every $f\in J$ we have $\Theta(f)(M_a)=0$ and so, by Remark \ref{rem28},  
\begin{align*}
 \du{a}{\gl_xf}=0 \quad \text{for almost all }\; x.
\end{align*}
Thus, for all $g \in L^1(G)$, we have 
$$\int_G g(x)\du{a}{\gl_xf}dx=0.\vspace{-0.5ex}$$ Therefore
\begin{align*}
\du {a}{( g * f)} &= \int_G ( g * f)(y)a(y)dy 
= \int_G \left(\int_G g(x) f(x^{-1}y)dx\right)a(y)dy\\
& =\int_G g(x)\left(\int_G f(x^{-1}y)a(y)dy\right)dx=\int_G g(x)\du{a}{\gl_xf}dx=0 \,.
\end{align*}
Let $(g_i)$ be an approximate unit for $L^1(G)$. 
Then 
$$\du{a}{f}=\lim \du{a}{g_i*f}=0,$$
and hence $a\in J^\bot$. 
\end{proof}

\begin{proposition} \label{propeasy}
For every left ideal $J\subseteq L^1(G)$, we have 
\begin{equation}\label{eq_inclu}
\Bim(J^{\perp})\subseteq (\Ran J)^{\perp}.
\end{equation}
\end{proposition}
\begin{proof} 
Since the maps $\Theta(f)$ are weak-* continuous, it suffices, by Lemma \ref{l}, 
to show that if $a\in J^\bot$ and $s,t\in G$, 
then 
$\Theta(f)(\gl_s^*M_a\gl_t)=0$ for all $f\in J$. 
But this follows from Lemma \ref{easy}. 
\end{proof}

In the subsequent sections, we will show that equality holds in (\ref{eq_inclu}) 
when $G$ is weakly amenable discrete, compact or abelian.  
We do  not know whether equality holds in (\ref{eq_inclu}) for a general  locally compact group $G$;
in the next lemma, we establish a useful restricted version, which 
should be compared to \cite[Lemma 4.6]{akt}. 
We identify the annihilator $J^{\perp}$ of an ideal $J\subseteq L^1(G)$ with its image in the
masa $\cl D = \cl D_G$.

\begin{proposition}\label{reconstruct}
For every  left ideal $J\subseteq L^1(G)$, 
$$\Bim(J^{\perp})\cap\cl D = (\Ran J)^{\perp}\cap\cl D= J^{\perp}.$$
\end{proposition}

\begin{proof} 
Trivially, 
$J^{\perp}\subseteq \Bim(J^{\perp})\cap\cl D$, while, by Proposition \ref{propeasy},
$\Bim(J^{\perp})\cap\cl D\subseteq (\Ran J)^{\perp}\cap\cl D$.
It remains to show that if $M_a\in (\Ran J)^{\perp}\cap\cl D$, then 
$a\in J^\perp$. But this follows from Lemma \ref{easy}.
\end{proof}

\section{the abelian case}\label{s_ac}
 
In \cite{akt2}, we used Theorem 2.1 to investigate the relation between 
$\gs$-harmonic functionals (where $\sigma$ is a multiplier of the Fourier algebra)
and $\gs$-harmonic operators.

In this section we assume that $G$ is a second countable 
locally compact abelian group and 
we obtain the equality 
$$\Bim(J^{\perp})=(\Ran J)^{\perp}$$ for an ideal $J\subseteq L^1(G)$. 

For this, we use Theorem 2.1 for the dual group $\Gamma$. 
To see the connection, let $\mu $ be a probability measure on $G$ and let $\sigma$ be the Fourier 
transform of $\mu$, that is, $\sigma = \hat{\mu}$, where 
$\hat{\mu}(x) = \int_G \overline{x(r)} d\mu(r), \ x\in \Gamma.$  
As $L^1(G)$ is a convolution ideal in $M(G)$ and $A(\Gamma) = \{\hat{f} : f\in L^1(G)\}$, the function 
$\sigma$ is a multiplier of $A(\Gamma)$ which,  
since $\vn(\gG)$ is an abelian C*-algebra, is completely bounded (see, for example, \cite[Prop. 2.2.6]{er}). 
It is not hard to see that, in this case,
the space of $\mu$-harmonic functions on $G$ is identified with the space of 
$\check\sigma$-harmonic functionals on $A(\Gamma)$ (here $\check\sigma(t)=\gs(t\an)$),
via the dual of the Fourier transform. In \cite{akt2}, we used Theorem 2.1 to investigate the relation between 
$\gs$-harmonic functionals (where $\sigma$ is a multiplier of the Fourier algebra)
and $\gs$-harmonic operators.

In this section,  we consider ideals  both of $A(\Gamma)$ and of $L^1(G)$. To improve clarity, 
if $I$ is an ideal of $A(\Gamma)$,
we will denote by  $\Bim_{\cl D_{\Gamma}}(I^{\perp})$  the $D_{\Gamma}$-bimodule of $\cl B(L^2(\Gamma))$ 
generated by the annihilator $I^{\perp}$ of $I$ in $\vn(\Gamma)$, while, if $I$ is an ideal of $L^1(G)$ 
we will denote by   $\Bim_{\vn (G)}(I^{\perp})$  the $\vn(G)$-bimodule  of  $\cl B(L^2(G))$ generated 
by the multiplication operators with symbols in the annihilator  $I^{\perp}$ of $I$  in $L^{\infty}(G)$.

For  a closed ideal $J\subseteq L^1(G)$, we wish to prove the equality 
\begin{equation}
(\Ran J)^\bot = \Bim\mbox{}_{\vn(G)}(J^\bot)  \, . \label{eq}
\end{equation} 
Let $F:L^2(G)\to L^2(\Gamma)$ be the unitary operator such that 
$F(f)= \hat f$, $f\in L^1(G)\cap L^2(G)$, and 
\[\Phi: \cl B(L^2(G)) \to\cl B( L^2(\gG)), \ \ \  \Phi(T)= FTF^{-1}.\]
It is clear that
$\Phi(\cl D_{G}) =\vn(\gG)$ and $\Phi(\vn(G))=\cl D_{\gG}$, and it is readily verified that 
\begin{align*}
\Phi\left(\Bim\mbox{}_{\vn(G)}(J^\bot)\right) = \Bim\mbox{}_{\cl D_\gG}\left(\Phi(J^\bot)\right)
\text{ and }
\Phi\left( (\Ran J)^\bot\right) = \Psi \left(\Ran J\right)^\bot,
\end{align*}
where $\Psi: \cl T(G) \to \cl T(\gG)$ denotes the predual of the map $\Phi\an$. 
Hence, (\ref{eq}) is equivalent to 
\begin{align}\label{psi2}
 (\Psi (\Ran J))^\bot =\Bim{}_{\cl D_{\gG}}(\Phi(J^\bot)),
\end{align} 
after identifying $J^\perp$ with its image in $\cl D_G$.

We will need the following lemma. 

\begin{lemma}\label{l_htg} 
Let $h\in\cl T(G)$ and $f\in L^1(G)$. Then
\[\Psi(\theta(f)(h)) =N(\phi(\hat f))\Psi(h),\]
where  $\phi$ denotes the map $\phi(u)(x)=u(x\an), \; x\in\gG$.
\end{lemma}
\proof
Since the maps $\Psi$ and $\theta(f)$ are linear and continuous on $\cl T(G)$, it suffices to prove the Lemma
when  $h(x,y)=\xi(x)\bar\eta(y)$, where $\xi,\eta$ are continuous with compact support. 
Note that, since $F : L^2(G) \to L^2(\Gamma)$ is a unitary operator, the map 
$F_2$, given on elementary tensors by $F_2(\xi\otimes\eta) = F(\xi)\otimes F(\eta)$, 
is a well-defined bounded linear map from $\cl T(G)$ into $\cl T(\Gamma)$. 
Since the function 
$(s,t,r)\to  h(sr^{-1},tr^{-1}) f(r)$ is in $L^1(G\times G\times G),$
for $x,y\in \gG$ we have
\begin{align*}
F_2(\theta(f)(h))(x,y)& = 
\int_G\int_G \overline {x(s)} \overline{y(t)} (\theta(f)h)(s,t) dsdt\\
& = 
\int_G\int_G\int_G \overline{x(s)} \overline{y(t)} h(sr^{-1},tr^{-1}) f(r) drdsdt\\
& = 
\int_G\int_G\int_G \overline{x(sr)y(tr)} h(s,t) f(r) drdsdt\\
& = 
\int_G\int_G\int_G \overline{x(s)x(r)y(t)y(r)} h(s,t) f(r) drdsdt\\
& = 
\int_G\int_G\int_G \overline{ x(s) y(t) (xy)(r)}
h(s,t) f(r) drdsdt\\
& = 
\int_G\int_G\overline{ x(s) y(t)}
\left(\int_G \overline{(xy)(r)} f(r)dr\right) h(s,t) dsdt\\
& = 
\hat{f}(xy) \int_G\int_G\overline{ x(s) y(t)}
h(s,t) dsdt\\
& = 
\hat f(xy) F_2(h)(x,y) .
\end{align*}
But it is not hard to verify that, for all such $\xi,\eta$,  we have
\begin{align} \label{psi}
\Psi(\xi\ot\eta)(x,y) &=(\hat\xi\ot\overline{\hat\eta})(x,y) =
F_2(\xi\ot\bar\eta)(x,y\an)\nonumber \\
\text{and so }\; \Psi(h)(x,y) & =F_2(h)(x,y\an) 
\end{align} 
for all $h\in \cl T(G)$. 
Thus the previous equality gives
\begin{align*}
\Psi(\theta(f)(h))(x,y) &=F_2(\theta(f)(h))(x,y\an) = 
\hat{f}(xy\an)F_2(h)(x,y\an) \\ & =\phi(\hat{f})(yx\an)\Psi(h)(x,y)
\end{align*}
i.e. $\; \Psi(\theta(f)(h)) =N(\phi(\hat f))\Psi(h). \qquad\Box$

\medskip

An operator $T\in\cl B(L^2(\Gamma))$ is in $(\Psi  (\Ran J)) ^\bot$ if and only if  
$\du{T}{\Psi(\theta(f)h)}_t$ $= 0$ for all $f\in J$ and $h\in\cl T(G)$.  
It follows from Lemma \ref{l_htg} that this is equivalent to  the statement that
$\du{T}{N(\phi(\hat f))\Psi(h)}_t=0$ for all $f\in J$ and $h\in\cl T(G)$. Noting that 
$\Psi$ maps $\cl T(G)$ onto $\cl T(\Gamma)$, we obtain that $T$ is in $(\Psi (\Ran J)) ^\bot$ if and only if
it annihilates $N(\phi(\hat J))\cl T(\Gamma)$, i.e. 
if and only if $T$ is in $(\Sat \phi(\hat J))^\bot$. (Here, $\hat J=\{\hat f: f\in J\}$.)

We have thus shown that
\[
(\Psi (\Ran J)) ^\bot=  (\Sat \phi(\hat J))^\bot \, .
\] 
Using Theorem \ref{th_satlcg} for the ideal $\phi(\hat J)\subseteq A(\gG)$, we see that 
$$ (\Sat \phi(\hat J))^\bot=\Bim{}_{\cl D_{\gG}} (\phi(\hat J)^\bot)$$  
and so it follows that
\[
(\Psi (\Ran J)) ^\bot=  \Bim\mbox{}_{\cl D_{\gG}} (\phi(\hat J)^\bot) \, .
\] 
Thus the required equality (\ref{psi2}) becomes
\[
\Bim\mbox{}_{\cl D_{\gG}} (\phi(\hat J)^\bot)  =\Bim\mbox{}_{\cl D_{\gG}}(\Phi(J^\bot)).
\] 
It now suffices to prove that
\[
(\phi(\hat J))^\bot  =\Phi(J^\bot) \, .
\] 
We have 
\begin{align*}
\Phi(J^\bot) = \left\{\Phi(M_g): M_g \in\cl D_G,\,\int_Gg(s)f(s)ds=0\; \text{ for all } f\in J\right\} .
\end{align*}
On the other hand, using the fact
that $\vn(\gG)=\Phi(\cl D_G)$, we have that
\begin{align*}
(\phi(\hat J))^\bot &= \left\{T\in\vn(\gG):\du{T}{\phi(\hat f)}_A=0 \;  \text{ for all } f\in J\right\} \\
&= \left\{\Phi(M_g): M_g \in\cl D_G, \,\du{\Phi(M_g)}{\phi(\hat f)}_A=0 \; \text{ for all } f\in J\right\},
\end{align*}
where $\du{\cdot}{\cdot}_A$ denotes the Banach space duality between $\vn(\gG)$ and $A(\gG)$.

Thus, it suffices to prove that, for any $f\in L^1(G)$ and $g\in L^\infty(G)$ the equality
\begin{align}\label{star}
\du{\Phi(M_g)}{\phi(\hat f)}_A=\int_Gg(s)f(s)ds 
\end{align}
holds. Fix $f\in L^1(G)$ and note that both sides of (\ref{star}) are  linear and w*-continuous functions of $g$. 
Since $L^\infty(G)$ is the w*-closed
linear span of the set $\{x:x\in \gG\}$ of characters,
it suffices to prove (\ref{star}) when $g$ is a character $ x$. Now $\Phi(M_{ x})=\gl_x$. 
Since $\du{\gl_x}{\hat f}_A =\hat f(x)$, we have
\begin{align*}
\du{\Phi(M_{ x})}{\phi(\hat f)}_A &= \du{\gl_x}{\phi(\hat f)}_A = \phi(\hat f)(x)=\hat f(x\an) \\
&= \int_Gf(s)\overline{x\an(s)}ds = \int_Gf(s)x(s)ds,
\end{align*}
as required.

\medskip

This concludes the proof of the following: 
\begin{proposition}\label{p_abel}
 Let $G$ be a locally compact abelian group. Then, for  any closed ideal $J\subseteq L^1(G)$, 
$$(\Ran J)^{\perp} = \Bim(J^{\perp}).$$
\end{proposition}


\section{The discrete case}\label{s_dc}

In this section we assume  that $G$ is discrete; in this case, the Haar measure coincides with the counting measure. 
We denote by $\delta_s$ the function on $G$ defined by $\delta_s(t)=1$ if $s=t$ and 
$\delta_s(t)=0$ if $s\neq t$; note that $\{\delta_s\}_{s \in G}$ is an orthonormal basis of $L^2(G)$.
Let   $X$ be an operator in $ \cl B(L^2(G))$. We denote by  {$[X(s, t)]$} be the matrix of $X$  
 with respect to the basis $\{\delta_s\}_{s \in G}$. The diagonal $D(X)$ of $X$ is the  operator on  $ L^2(G)$
 whose matrix with respect to the basis 
$\{\delta_s\}_{s \in G}$ is given by $D(X)(s,t)=0$ if $s\neq t$ and $D(X)(s, t)=X(s, t)$ if $s=t$.
For $t\in G$, we denote by $D_t(X)$ the $t$-th diagonal of $X$, given by $D_t(X) = \lambda_tD(\lambda_{t^{-1}}X)$.
Note that the maps $X\mapsto D_t(X)$ are weak-*continuous and linear.


Also note that $D_r(X)=S_{N(\gd_r)}(X)$. Indeed, 
\begin{align*}
S_{N(\gd_r)}([X(s, t)]) = [\gd_r(ts\an)(X(s, t)] = \left [ \begin{cases} X(s, rs), & t=rs \\ 0, & t\ne rs \end{cases} \right] 
\end{align*}
Thus, if $u:G\to\bb C$ is a {\em finitely supported} function, then $S_{N(u)}(X)$
 is a linear combination of diagonals of $X$.  

Suppose that the group $G$ is weakly amenable  in the sense of \cite{CowH}. This means that there exists a net
$\{u_i\}_{i \in I}$ consisting of finitely supported elements of $A(G)$ and a positive constant $L$ such that 
$\nor{u_i}_{mcb}\le L$ for all $i$ and $u_i(s)\to 1$ for all $s\in G$ (here $\nor{u_i}_{mcb}$ is the 
completely bounded norm of $u_i$ as a multiplier of $A(G)$, or equivalently of the Schur multiplier $S_{N(u_i)}$). It follows that 
for each $h\in \cl T(G)$ we have 
$$\nor{N(u_i)h}_t\le L \nor{h}_t \quad \text{for all }\; i.$$

\begin{proposition}\label{diag}
Let $G$ be a weakly amenable (discrete) group. Then each $A\in \cl B(L^2(G))$ is in the weak-* closed linear span 
of its diagonals.  
\end{proposition}
\proof
Recall the diagonals of $A$ are $S_{N(\gd_t)}(A)$, $t\in G$. Thus if $h\in \cl T(G)$ annihilates all diagonals of $A$, then
\begin{align*}
0 = \du{S_{N(\gd_t)}(A)}{h} = \du{A}{N(\gd_t)h} \quad\text{for all } \; t\in G.
\end{align*}
But $N(\gd_t)h(s,r)=\gd_t(rs\an) h(s,r) = h(s, ts)$ when $r=ts$ and $=0$ otherwise. 
Thus $A$ must annihilate all the diagonals of $h$.
If we prove that $h$ is  in the trace-norm closed linear span of its diagonals, it will follow that 
$\du{A}{h} =0$, as required.  

It thus remains to prove that $h$ is  in the trace-norm closed linear span of its diagonals. For this, 
observe first that given $\eps>0$ there is an $h_\eps\in \cl T(G)$, supported on finitely many diagonals, such that
$\nor{h-h_\eps}_t<\eps$ (it suffices to take $h_\eps$ of the form $php$ where $p$ is the projection on the span of
a suitably large but finite  subset $\{\gd_t: t\in F\}$ , since such projections tend strongly to the identity). 

But note that $$\lim_i\nor{N(u_i)h_\eps-h_\eps}_t=0\, .$$
This is because on each of the finitely many nonzero diagonals $D_t(h_\eps)$ we have 
$N(u_i)D_t(h_\eps)= u_i(t)D_t(h_\eps)$, hence 
$\nor{N(u_i)D_t(h_\eps)-D_t(h_\eps)}_t=$ \\ $=|u_i(t)-1|\nor{D_t(h_\eps)}_t$,
 and $u_i(t)\to 1$.
Therefore we can choose $i_0$ such that $\nor{N(u_i)h_\eps-h\eps}_t<\eps$ for all $i\ge i_0$. 

Thus finally we have, for all $i\ge i_0$,
\begin{align*}
\nor{N(u_i)h-h}_t & \le \nor{N(u_i)(h-h_\eps)}_t+\nor{N(u_i)h_\eps-h_\eps}_t+\nor{h_\eps-h}_t \\ 
& \le L\nor{h-h_\eps}_t+\nor{N(u_i)h_\eps-h_\eps}_t+\nor{h_\eps-h}_t <L\eps +\eps +\eps \, .
\end{align*}
This shows that $h$ is in the trace-norm closed linear span of the family $\{N(u_i)h:i\in I\}$; but as observed above,
since each $u_i$ is finitely supported, each $N(u_i)h$ is a linear combination of diagonals {\em of $h$}. 
This proves the claim and concludes the proof of the Proposition. \qed

\begin{lemma}\label{Ldiag}
Let  $G$ be a discrete group and $J\subseteq L^1(G)$ be a closed left ideal. 
If $X \in (\Ran J)^\bot$, then $D_t(X) \in (\Ran J)^\bot$, for all $t \in G$.
\end{lemma}
\begin{proof}
A direct calculation shows that $$D(\rho_sX\rho_{s^*})=\rho_sD(X)\rho_{s^*}.$$
It follows by the weak-* continuity of $D$ that 
$$\Theta(f)D(X)=D(\Theta(f)(X))$$ 
for $f \in L^1(G)$.
The conclusion follows from Lemma \ref{l}.
\end{proof}

\begin{proposition}\label{p_disrc}
Let $G$ be a discrete weakly amenable group and $J\subseteq L^1(G)$ be a closed left ideal. Then
$$(\Ran J)^{\perp} = \Bim(J^{\perp}).$$
\end{proposition}
\begin{proof}
Let $X \in (\Ran J)^\perp$.  Since $(\Ran J)^\perp$ is a $\vn (G)$-bimodule $\gl_{t^{-1}}X\in (\Ran J)^\perp$ and 
it follows from Lemma \ref{Ldiag} that $D(\lambda_{t^{-1}}X)\in (\Ran J)^{\perp}$.
Now, $D(\lambda_{t^{-1}}X)=M_{a_t}$ for some $a_t \in \ell^{\infty}(G)$. It follows from 
Lemma \ref{easy} that $a_t \in J^{\bot}$, and hence $D_t(X) \in \Bim(J^{\bot})$.
Since the  operator $X$ is in the weak-* closed linear span of its diagonals (Proposition \ref{diag}), 
we obtain that $X \in \Bim(J^\bot)$.
 
By Proposition \ref{propeasy}, the proof is complete.
\end{proof}

\begin{remark}
In a previous version of this paper we claimed that Proposition \ref{p_disrc} holds in any discrete group. 
We wish to thank J. Crann and M. Neufang who pointed out that our argument was incomplete. 
\end{remark}


\section{The compact case}\label{s_cc}

In this section we assume that $G$ is compact. We denote by $\widehat{G}$ the unitary dual of $G$, 
that is, the set of all (equivalence classes of) irreducible representations. 
If $\pi \in\widehat G$, 
we denote by $H_\pi$ the space of the representation $\pi$, and by $d_{\pi}$ its dimension.
Suppose that for each irreducible representation  $(\pi, H_\pi)$ of $G$ we are given a subspace $E_\pi\subseteq H_\pi$
(possibly trivial). If $E_\pi\neq \{0\}$ choose an orthonormal  basis $e_1,\dots, e_{s_\pi}$ of $E_\pi$  
and extend it to an orthonormal 
basis $e_1,\dots, e_{d_\pi}$ of $H_\pi$. If $E_\pi= \{0\}$ let $e_1,\dots, e_{d_\pi}$ be an orthonormal  basis  of $H_\pi$. 
For $\pi \in \widehat{G}$, we denote by $\pi_{ij}$,  $1\le i,j\le d_\pi$ the matrix coefficients 
of the representation $\pi$ with respect to the basis $e_1,\dots, e_{d_\pi}$ of $H_\pi$;
thus, 
\begin{equation}
\pi_{ij}(s)=(\pi(s)e_j, e_i), \ \ \ s\in G, \ i,j = 1,\dots,d_{\pi}. \label{eq2}
\end{equation}

Let $E=\{E_{\pi}\}_{\pi \in \widehat{G}}$ and consider the set 
\begin{align*}
J(E) &= \overline{\sspp{\left\{\pi_{ij}:1\le i\le d_\pi,1\le j\le s_\pi, \pi\in\widehat G, \;E_{\pi}\neq \{0\}\right\}}}^{\nor{\cdot}_1},  
\end{align*}
where $\|\cdot\|_1$ is the $L^1(G)$ norm.
Clearly, $J(E)$ is a closed left ideal of $L^1(G)$, being invariant under left translations.   
Conversely, every closed left ideal of $L^1(G)$ is of this form \cite[38.13]{hr} for some $E=\{E_{\pi}\}_{\pi \in \widehat{G}}$.

Denoting  by $J(E)^{\perp}$ the  annihilator in $L^{\infty}(G)$, we have: 

\begin{proposition}\label{41}
The space $J(E)^{\perp}$  is the $w^*$-closure of the linear span of 
\begin{align*}
 \cl S: =&
\{\overline{\pi'_{ij}}:1\le i\le d_{\pi'},s_{\pi'}< j\le d_{\pi'}\, ,  \;E_{\pi'}\neq \{0\}\} \; \cup \; \\ 
&\cup \;\{\overline{\pi'_{ij}}:1\le i,j\le d_{\pi'}\, , \;E_{\pi'}= \{0\}\}.
\end{align*}
\end{proposition}

\proof 
Let $\pi' \in \widehat{G}$ be such that $E_{\pi'}\neq \{0\}$ and $1\le k\le d_{\pi'},s_{\pi'}< l\le d_{\pi'}$.
Let $\pi \in \widehat{G}$ be such that $E_{\pi}\neq \{0\}$ and $1\le i\le d_\pi,1\le j\le s_\pi$. 
If $\pi'$ is not equivalent to $\pi$, then $\int\overline{\pi'_{kl}(t)}{\pi_{ij}(t)}dt=0$ for all $k,l$
by the Schur orthogonality relations \cite[5.8]{follharm}. 
If $\pi'$ is equivalent to $ \pi$, 
then  $\int\overline{\pi'_{kl}(t)}{\pi_{ij}(t)}dt=0$ for all $k$ 
since $j\ne l$. Moreover, it is clear that 
$$\bigcup \;\{\overline{\pi'_{ij}}:1\le i,j\le d_{\pi'}\, , \;E_{\pi'}= \{0\}\}\subseteq J(E)^{\perp}.$$
Hence $\cl S\subseteq J(E)^{\perp}$.

For the reverse containment, we show that the preannihilator $\cl S_{\bot}$ is contained in $J(E)$.
Now $\cl S_{\bot}$ is a closed left ideal in $L^1(G)$, since the linear span of $\cl S$
 is  invariant under left translations.
 Take $f\in \cl S_{\bot}$. Let $(g_\nu)$ be an approximate  unit for $L^1(G)$  
consisting of  functions in $L^2(G)$
and set  $f_\nu=g_\nu*f$; so $f_\nu\in L^2(G)$ and $\nor{f-f_\nu}_1\to 0$. Since each $f_\nu$ is in 
$\cl S_{\bot}$, it is orthogonal (in the $L^2(G)$ sense)
to $\pi'_{ij}$'s whose conjugate  generate $\cl S$ and hence, by the Peter-Weyl theorem, 
it belongs to the $L^2(G)$ closed span 
of the remaining $\pi'_{ij}$'s, that is,  to the  closure of \vspace{-0.5ex}
$$\sspp{\{\pi'_{ij}: 1\le i\le d_{\pi'}, 1\le j\le s_{\pi'} \, , \;E_{\pi'}\neq \{0\}\}} \vspace{-0.5ex}$$ 
in $L^2(G)$.
But this closure of this set is contained in its $L^1(G)$ closure, which coincides with $J(E)$.  
Thus $f_\nu\in J(E)$ for each $\nu$, and so $f\in J(E)$. \qed

\begin{remark} We would like  to observe that the above Proposition may be proved using the theory of 
strong M-bases in Banach spaces: 

 Let $X$ be a Banach space.
A family of vectors $(u_i)$ is called a Markushevich basis or an $M$-basis of $X$ \cite[Definition 1.7]{hmsvz}
if there exists a family $(u'_i)$ in the dual $X^*$  of $X$  such that 
 \begin{enumerate}
  \item    $\du{u'_i} {u_j}_X=\delta_{ij}$, where $\du{\cdot}{\cdot}$ is the pairing between $X^*$  and $X$
    
  \item  $\overline{ \sspp\{u_i\}}^{\nor{\cdot}}=X$
    
  \item $ \overline{\sspp\{u'_i\}}^{w*}=X^*$.
   \end{enumerate}
The family $(u_i)$  is called a strong $M$-basis \cite[Definition 1.32]{hmsvz} if for every $x \in X$  we have 
$$x \in \overline{\sspp\{u_i: \du{u'_i}{x}_X\neq 0\}}^{\nor{\cdot}}.$$

  It follows from \cite[2.9.3]{ed} that the family $\{\pi_{ij}: 1\leq i, j\leq d_{\pi}, \pi \in \widehat{G}\}$ as 
  defined in (\ref{eq2}) 
 is a strong $M$-basis of the space  $L^1(G)$. 
 Proposition \ref{41} now follows from \cite[Proposition 1.35]{hmsvz}.
\end{remark}

By the Peter-Weyl theorem (see for example \cite[Theorem 5.12]{follharm}), $L^2(G)$ is the orthogonal direct sum 
 \[
 L^2(G)=\bigoplus_{\pi\in\widehat G}\cl E_\pi \vspace{-1ex}
 \]
where \vspace{-1ex}
 \[\cl E_\pi =\sspp{ \{\sqrt{d_\pi}\pi_{ij}, 1\le i,j\le d_\pi\}}.\]
Moreover, $ \sqrt{d_\pi}\pi_{ij}, 1\le i,j\le d_\pi$ is an orthonormal basis of $\cl E_\pi$. 
If $\pi \in \widehat{G}$, denote by  $P_\pi\in \cl B(L^2(G))$  the orthogonal projection onto
$\cl E_\pi$.

With respect to this decomposition, each $T\in\cl B(L^2(G))$ corresponds to an infinite  matrix 
$T=[T_{\pi,\pi'}]$ of operators $T_{\pi,\pi'}\in \cl B(\cl E_\pi',\cl E_{\pi})$
which act on finite dimensional spaces, where $T_{\pi,\pi'}=P_{\pi}TP_{\pi'}$.

\medskip
\begin{remark}\label{r_vn}
If $\pi \in \widehat{G}$ then  $P_\pi \in\vn (G)$.
\end{remark}
Indeed, since  $\cl E_\pi$ is $\rho_s$ invariant, we have
$P_\pi\rho_s=\rho_sP_\pi$ for all $s\in G$. 

\begin{remark} \label{two}
An operator $T$ is in $(\Ran J)^\bot$ (resp. $\Bim(J^\bot)$) if and only if 
$T_{\pi,\pi'}$  is in $(\Ran J)^\bot$ (resp. $\Bim(J^\bot)$), for all $\pi,\pi'\in \hat{G}$.
\end{remark} 
\begin{proof}  
Since  $(\Ran J)^\bot$ is a $\vn(G)$-bimodule, if $T\in (\Ran J)^\bot$ then, by Remark \ref{r_vn},
$T_{\pi,\pi'}=P_\pi TP_{\pi'}$ is in $(\Ran J)^\bot$. Conversely, if $T_{\pi,\pi'}\in (\Ran J)^\bot$
for all $\pi,\pi'\in \hat{G}$
then, since $T$ is in the weak-*closed linear span of $\{T_{\pi,\pi'}:\pi,\pi'\in\widehat G\}$ and $(\Ran J)^\bot$
is a weak-*closed subspace, it follows that  $T\in (\Ran J)^\bot$. 

The proof for $\Bim(J^\bot)$ is identical.
\end{proof}

\begin{theorem}\label{th_mainc}  
Let $G$ be a compact group and $J\subseteq L^1(G)$ be a closed left ideal. Then
$$(\Ran J)^\perp = \Bim(J^\perp).$$
\end{theorem}

\begin{proof}  
By Proposition \ref{propeasy}, it is enough to show that, 
if an operator $T$ is in $(\Ran J)^\bot$, then $T\in\Bim(J^\bot)$.
By Remark \ref{two}, it suffices to prove that, for all  $\pi,\pi'\in\widehat G$, we have  $T_{\pi,\pi'}\in\Bim(J^\bot)$.

Fix $\pi,\pi'\in\widehat G$ and write $P:=P_{\pi}$ and $Q:=P_{\pi'}$ to simplify notation.
We have to prove that $PTQ\in\Bim(J^\bot)$.
Recall that the linear span of the set
\[ \left\{M_{\overline{\pi_{ij}}}\gl_s: \pi\in \widehat G, 1\le i, j\le d_\pi,  s\in G\right\}\]
is a *-algebra with trivial commutant, it is weak-*dense in $\cl B(L^2(G))$. 
It follows that the linear span of the set 
\begin{equation}\label{eq_PM}
\{PM_{\overline{\pi_{ij}}}\gl_sQ: \pi\in \widehat G, 1\le i, j\le d_\pi,  s\in G\} \tag{*}
\end{equation}
is weak-*dense in $\cl B(QL^2(G), PL^2(G))$. Since  $\cl B(QL^2(G), PL^2(G))$ is finite-dimensional, we have 
\[ \sspp\{PM_{\overline{\pi_{ij}}}\gl_sQ:\pi\in \widehat G, 1\le i, j\le d_\pi, s\in G\} = \cl B(QL^2(G), PL^2(G)).\]

From the generating set (\ref{eq_PM}) we choose an algebraic basis $\{PM_k\gl_{s_k}Q: 1\le k\le m\}$ of 
$\cl B(QL^2(G), PL^2(G))$, where each $M_k$ is  $M_{\overline{\pi_{ij}}}$ for some $\pi\in\widehat G$ and some
$1\le i, j\le d_\pi$.
There are scalars $c_k$ such that 
\begin{equation}\label{ptq}
PTQ = \sum_{k=1}^mc_kPM_k\gl_{s_k}Q \, .
\end{equation}
We will show that  the only nonzero terms in this sum are those for which $M_k = M_{\overline{\pi_{ij}}}$, for some 
$\pi$, $i$, $j$, 
where, either $E_{\pi}=\{0\}$, or $E_{\pi}\neq \{0\}$ and $s_\pi < j\le d_\pi$. 
{Since such terms are in $\Bim(J^\bot)$ it will follow that  
 $PTQ \in \Bim(J^\bot)$, thus completing the proof.}  

For a continuous function $f$ we have (recalling that $\Theta(f)$ is a $\vn(G)$-bimodule map)
\begin{equation}\label{gptq}
\Theta(f)(PTQ) = \sum_{k=1}^mc_kP\Theta(f)(M_k)\gl_{s_k}Q \, . 
 \end{equation}
Fix $k\in \{1,\dots,m\}$, and let $\pi_{ij}$ be such that $M_k = M_{\pi_{ij}}$. 
Then
\begin{align*}
\Theta(f)(M_k)=\Theta(f)(M_{\overline{\pi_{ij}}}) =\int_G f(r)(\rho_rM_{\overline{\pi_{ij}}}\rho_r^*)dr = M_g  
\end{align*}
where  
$ g(x)=\int_G f(r)\overline{\pi_{ij}}(xr)dr$ (Remark \ref{rem28}),  that is 
\begin{align*}
g(x) &=  \sca{f, \sum_k \pi_{ik}(x)\pi_{kj}}=\sum_k \overline{\pi_{ik}(x)}\sca{f, \pi_{kj}}.
\end{align*}

Let $\pi'\in\widehat G$ be such that $E_{\pi'}\neq \{0\}$ and choose  $f=d_{\pi'}\pi'_{nn}$ where $1\le n\le s_{\pi'}$. Then, by the orthogonality relations, 
\begin{align*}
g(x) = \sum_k \overline{\pi_{ik}(x)}\gd_{nk}\gd_{nj} \gd_{\pi\pi'}= \overline{\pi_{in}(x)}\gd_{nj}\gd_{\pi\pi'} 
\end{align*}
It follows that  
\begin{align*}
\Theta(f)(M_{\overline{\pi_{ij}}}) =\Theta(d_{\pi'}\pi'_{nn})(M_{\overline{\pi_{ij}}}) 
=  M_{\overline{\pi_{in}}}\gd_{nj}\gd_{\pi\pi'}=M_{\overline{\pi_{ij}}}\gd_{nj}\gd_{\pi\pi'} \, .
\end{align*}
Hence all the monomials 
in the expression (\ref{gptq}) for $\Theta(f)(PTQ)$ must vanish, except when $\pi=\pi'$ and $j=n$, in which case they are left unchanged. 
Thus (\ref{gptq}) gives
\begin{equation}\label{eq_or}
\Theta(f)(PTQ) = \sum_kc_kPM_k\gl_{s_k}Q,
\end{equation}
the summation being over those  $k$ for which $M_k=M_{\overline{\pi'_{in}}}$.

Now $f\in J$ since $1\le n\le s_{\pi'}$; thus, by Lemma \ref{l}, $\Theta(f)(PTQ)=0$ 
and therefore the sum (\ref{eq_or}) must vanish. 
But the monomials $PM_k\gl_{s_k}Q$ are linearly independent (they were chosen from an algebraic  basis) 
and so each term must vanish.

Thus, for all $\pi'_{ij}$ with $E_{\pi'}\neq \{0\}$, $1\le i\le d_{\pi'}$ and  $1\le j\le s_{\pi'}$, all 
terms of the form $c_kPM_{\overline{\pi'_{ij}}}\gl_{s_k}Q$ must vanish 
in the sum (\ref{ptq}).
Therefore in this sum  the only nonzero 
terms 
remaining are of the form $c_kPM_k\gl_{s_k}Q$
where $M_k=M_{\overline{\pi_{ij}}}$ for some $\pi_{ij}$ with $E_{\pi}\neq \{0\}$ and $s_\pi<j\le d_\pi$   
or for some $\pi_{ij}$ with $E_{\pi}= \{0\}$. 
By Proposition \ref{41}, these are in $\Bim(J^\bot)$, hence $PTQ\in\Bim(J^\bot)$ 
as required.
 \end{proof}


\section{Jointly Harmonic Operators}\label{s_jho}

In this section $G$ is a locally compact group.
If $\mu$ is a probability measure on $G$, let $P_\mu$ be the map on $L^\infty(G)$ given by
$$(P_\mu\phi)(s) = \int_G\phi(st)d\mu(t) \, .$$
A function $\phi$ is called {\em $\mu$-harmonic} \cite{who,furst} if it satisfies the relation 
$$P_\mu\phi=\phi.$$ 

More generally, given a set $\Lambda\subseteq M(G)$ (not necessarily consisting of probability measures)
we define the set $\cl H(\Lambda)$ of {\em jointly $\Lambda$-harmonic} functions by letting 
$$\cl H(\Lambda):= \left\{\phi\in L^\infty(G): P_\mu\phi=\phi \text{ for all } \mu\in\Lambda\right\}.$$
Note that $\cl H(\Lambda)$ is a weak-* closed linear subspace of $L^\infty(G)$. 
The preannihilator of $\cl H(\Lambda)$ in $L^1(G)$ is 
 $$J_{\Lambda} : =\overline{\sspp\{f\ast\mu -f: f\in L^1(G), \mu \in \Lambda \}}$$
\cite[page 8]{cl}.
Since $\cl H(\Lambda)$ is  invariant under left translations, the space 
 $J_{\Lambda} $ is a left ideal in  $L^1(G)$.

The map $\Theta(\mu)$ extends $P_\mu$ (under the natural identification of $L^{\infty}(G)$  with $\cl D_G$): for every $\phi\in L^\infty(G)$ and any $\mu\in M(G)$, we have  
$$ \Theta(\mu)(M_\phi) = M_{P_\mu\phi}\, $$
and so $\phi\in\cl H(\Lambda)$ if and only if $ \Theta(\mu)(M_\phi) = M_\phi$ for all $\mu\in\Lambda$.
It is therefore natural to define the set $\widetilde{\cl H}(\Lambda)$ of all {\em jointly $\Lambda$-harmonic operators} 
by letting 
$$\widetilde{\cl H}(\Lambda):=\{T\in\cl B(L^2(G)): \Theta(\mu)(T)=T \text{ for all } \mu\in\Lambda\}.$$
This weak-* closed subspace of $\cl B(L^2(G))$ is a 
$\vn(G)$-bimodule (because $\Theta(\mu)$ is a $\vn(G)$-bimodule map for every $\mu$) and it contains
$\{M_a: a\in\cl H(\Lambda)\}$; 
hence it contains $\Bim(  \cl H(\Lambda))$. 

\begin{theorem} \label{71} 
If $\Lambda\subseteq M(G)$ then
 $$\widetilde{\cl H}(\Lambda) = (\Ran J_{\Lambda})^{\perp} \, .$$
 \end{theorem}
 
\begin{proof}
Recall that $\Ran J_{\Lambda}$ is the closed linear span of $\theta(u)h$ where $u\in J_\Lambda$ 
 and $h \in \cl T(G)$. If $u=f\ast\mu-f$ where  $f \in L^{1}(G)$, $\mu\in\Lambda$ and
 $T \in \cl B(L^2(G))$ then
  $$\langle T, \theta(u)h\rangle_t =\langle \Theta(f) \Theta(\mu-\delta_e)T, h \rangle_t.$$
By Lemma \ref{l}, 
$T\in (\Ran J_{\Lambda})^{\perp}$ if and only if
\begin{equation}\label{eq_mudel}
\Theta(f) \Theta(\mu-\delta_e)T=0, \ \ \ f \in L^1(G), \ \mu \in \Lambda.
\end{equation}
Since $\Theta$ is the integral of a 
bounded representation of $G$, namely $\mathop{\rm Ad}\rho$, it is 
a non-degenerate representation of $L^1(G)$.
Thus, (\ref{eq_mudel}) holds true of and only if
$$\Theta(\mu-\delta_e)T=0 \; \text{ for all }\; \mu \in \Lambda$$
i.e. if and only if $T \in \widetilde{\cl H}(\Lambda).$
\end{proof}

Theorems \ref{71} and \ref{th_mainc} and Propositions \ref{p_disrc} and \ref{p_abel} imply the following corollary.

\begin{corollary}\label{c_ifh}
Let $G$ be a locally compact group such that 
$(\Ran J_{\Lambda})^{\perp}=\Bim( J_{\Lambda}^\perp)$ for every closed left ideal $J$ of $L^1(G)$. 
Then 
\begin{equation}\label{eq_widet}
\widetilde{\cl H}(\Lambda) = \Bim(  \cl H(\Lambda)).
\end{equation}
In particular, (\ref{eq_widet}) holds true if $G$ is abelian, or weakly amenable discrete, or compact.
\end{corollary}


\section{The non-commutative Poisson boundary}

In this section, we discuss the case where $\Lambda$ is a singleton consisting of a probability measure, say $\mu$.
There exists a norm one projection $\widetilde{\cl E}$ on $\widetilde{\cl H}(\mu) := \widetilde{\cl H}(\Lambda)$
given by a pointwise-weak* limit of convex combinations of iterates of $\Theta(\mu)$. 
The noncommutative Poisson boundary of $\mu$, denoted by $\widetilde{\cl H}_\mu$, is defined  to be the space 
$\widetilde{\cl H}(\mu)$, equipped with the unique von Neumann algebra structure 
defined through the Choi-Effros product $\diamond$  given by $T\diamond S= \widetilde{\cl E}(TS)$ \cite{izadv}. 
The space $\cl H(\mu) := \cl H(\Lambda)$ is closed under $\diamond$ and therefore  
is a von Neumann subalgebra of $\widetilde{\cl H}_{\mu}$ denoted by $\cl H_{\mu}$. 

Thus $\widetilde{\cl H}(\mu)$ is an injective weak* closed operator system, 
and in fact so is its subspace $\cl H(\mu) $
(it is the range of a contractive projection from $\cl D$). 
Moreover, $\cl H(\mu)$ admits a natural action $\ga$ of $G$ 
by weak-* continuous unital completely positive isometries, given by the restriction of the action of $G$
on $L^\infty(G)$ by left translation:  $(\ga_s\phi)(t)=\phi(s^{-1} t)$ (the space  $\cl H(\mu)$ is invariant under
translation because $P_\mu$ commutes with each $\ga_s$).

We wish to show that the operator system $\widetilde{\cl H}(\mu)$ is isomorphic, as a dual operator system,
to the operator system crossed product $G\rtimes_\ga \cl H(\mu)$, which we now define:   

Let $\cl M$ be a dual operator system, and let $s\to\ga_s$ be an action of $G$ on $\cl M$ by  
weak-* continuous unital completely positive isometries. The action 
is encoded by the map 
\begin{align*}
\tilde\ga: \cl M\to L^\infty(G,\cl M): v\to (\ga\an_s(v))_{s\in G} \, , 
\end{align*}
which is a unital completely positive isometry. 
Let $\cl B:=\cl B(L^2(G))$ and identify $L^\infty(G,\cl M)$ with $L^\infty(G)\bar\ot \cl M\subseteq \cl B\bar\ot\cl M$.
We also have a map $$G\to \cl B\bar\ot\cl M : s\to\tilde\gl_s:=\gl_s\ot I\, .$$ 
\begin{definition}
The {\em crossed product} $G\rtimes_\ga\cl M$ is defined to be the subspace of 
$\cl B\bar\ot\cl M$     
generated by $\tilde\ga(\cl M)\cdot\tilde\gl(G)$:  it is the weak* closed space 
$$
G\rtimes_\ga\cl M:= \overline{\sspp\{\tilde\ga(v)\tilde\gl_s, \; v \in \cl M, s\in G\}}^{w*}
\subseteq \cl B\bar\ot\cl M.
$$
\end{definition}
\begin{remark}
The crossed product is independent of the representation of $\cl M$ as a weak*-closed operator subsystem of
some $\cl B(H)$. This is a general fact (see \cite{akt3}).  However in case $\cl M$ is additionally an 
{\em injective} operator system (as in the case $\cl{M=H}(\mu)$ considered here), 
it follows from the 
well known corresponding result for von Neumann algebra crossed products \cite[Theorem X.1.7]{tak}. 
This is because $\cl M$
admits a unique von Neumann algebra structure, $\cl N$ say, induced by the Choi-Effros product 
and its original operator space structure. Then $G\rtimes_\ga\cl M$ is isomorphic, as a dual operator system,
to the von Neumann algebra crossed product  $G\rtimes\cl N$, which does not depend on the representation
of $\cl N$ on Hilbert space.
\end{remark}

Let $V\in\cl B\bar\otimes\cl B$ be the \emph{fundamental unitary},  given by 
\begin{align*}  
\quad (V\xi)(s,t) = \xi(st,t)\Delta(t)^{1/2}, \quad\xi\in L^2(G)\otimes L^2(G),
\end{align*}
and define \vspace{-1ex}
\begin{align*}
 \tilde\gG : \cl B\to\cl B\bar\otimes\cl B \quad 
\text{ by }\quad\tilde\gG (T) :=V(T\otimes I)V^*.  \vspace{-1ex}
\end{align*} 
Note that  $\tilde\gG$ is clearly a normal *-homomorphism and an isometry, hence a normal 
unital completely positive map.

\begin{proposition}\label{p_oneinc}
We have that $\tilde\gG\left(\Bim(\cl H(\mu))\right) = G\rtimes_{\ga}\cl H(\mu)$.
In particular, 
\begin{equation}\label{eq_incmuu}
G\rtimes_{\ga}\cl H(\mu)\subseteq\tilde\gG(\widetilde{\cl H}(\mu)).
\end{equation}
\end{proposition}

\begin{proof}
It is well-known (and not hard to verify) that 
$V(\gl_r\otimes I)= (\gl_r\otimes I)V$ for all  $r\in G$
and 
$(I\otimes M_f)V=V(I\otimes M_f)$ for all $f\in L^\infty(G)$.

Thus, $V\in\cl B\bar\ot\cl D$. 
It follows that 
\begin{align*}
\tilde\gG (T) &=V(T\otimes I)V^*\in\cl B\bar\ot\cl D,\quad\text{for all } \; T\in\cl B, \\
\text{and}\quad   
\tilde\gG(\gl_r) &=  \gl_r\otimes I =\tilde\gl_r, \quad\text{for all } \; r\in G\, . 
\end{align*}

If $\phi\in\cl H(\mu)$ and $s\in G$ then the element 
$(\tilde\ga\phi)(s)=\ga\an_s(\phi)$ of $\cl H(\mu)$ acts 
as a multiplication operator on $L^2(G)$ as follows:
$$((\tilde\ga\phi)(s)\eta)(t)=(\ga\an_s(\phi))(t)\eta(t)=\phi(st)\eta(t), \;\; \eta\in L^2(G), t\in G\, .$$ 
We claim that, for every $\phi\in\cl H(\mu)$ and $r\in G$, 
 \begin{equation}\label{eq_gG}
\tilde\gG(M_\phi\gl_r) = \tilde\ga(\phi)\tilde\gl_r\, . 
\end{equation} 
Now $\tilde\gG(M_\phi\gl_r)=\tilde\gG(M_\phi)\tilde\gl_r$ so it suffices to prove that 
$\tilde\gG(M_\phi) = \tilde\ga(\phi)$ or, equivalently, that $(\tilde\ga(\phi))V=V(M_\phi\ot I)$. 
Indeed, for all $\xi,\eta\in L^2(G)$ we have  
\begin{align*}
(\tilde\ga(\phi)V (\xi\ot \eta))(s,t) =(\ga_{s\an}\phi)(t) V (\xi\ot \eta))(s,t)
&=\phi(st) \xi(st)\eta(t)\Delta(t)^{1/2} \\
\text{and}\;\; (V(M_\phi\ot I)(\xi\ot \eta))(s,t) =(V(\phi\xi\ot\eta))(s,t)) &=(\phi\xi)(st)\eta(t)\Delta(t)^{1/2}  
\end{align*}
which proves the claim.

\medskip

By linearity and w*-continuity, 
\begin{align*}\tilde\gG\left(\overline{\sspp\{M_\phi\gl_r, \; \phi\in \cl H(\mu), r\in G\}}^{w*}\!\right) &=   \overline{\sspp\{\tilde\ga(\phi)\tilde\gl_r, \; \phi\in \cl H(\mu), r\in G\}}^{w*},\\
\text{i.e. }\quad
\tilde\gG\left(\Bim(\cl H(\mu))\right) &= G\rtimes_{\ga}\cl H(\mu)\, .
\end{align*} 

Since $\Bim(\cl H(\mu))\subseteq\widetilde{\cl H}(\mu)$, we have in particular 
$G\rtimes_{\ga}\cl H(\mu)\subseteq\tilde\gG(\widetilde{\cl H}(\mu))$.
\end{proof}

In case $G$ is weakly amenable discrete, 
compact or abelian, by Corollary \ref{c_ifh} we know that  
$\Bim(\cl H(\mu))=\widetilde{\cl H}(\mu)$. Therefore the previous Proposition yields:  

\begin{proposition}\label{crosssystems}
Assume that $G$ is weakly amenable discrete, compact or abelian.
Then $\tilde\Gamma$ is an isomorphism of dual operator spaces between 
$\widetilde{\cl H}(\mu)$ and the crossed product $G\rtimes_{\ga}\cl H(\mu)$. 
\end{proposition}
 
\begin{corollary}
Assume that $G$ is weakly amenable discrete, compact or abelian. Then the crossed product $G\rtimes_{\ga}\cl H(\mu)$ 
is an injective operator system.  
\end{corollary}

For $G$ weakly amenable discrete, compact or abelian  we obtain the following Corollary, established by Izumi in  \cite{izu} for discrete groups, by Jaworski and Neufang 
 for locally compact groups \cite{jn} and by Kalantar, Neufang and Ruan for locally compact quantum groups 
\cite{knr}. Analogous results were obtained in  \cite{sask} for  complex contractive measures.

Using these results, together with Theorem \ref{71} we obtain, for any locally compact group $G$, the equality  
$(\Ran J_\Lambda)^\perp = \Bim(J_\Lambda^\perp)$ when $\Lambda=\{\mu\}$ and $\mu$ is a probability measure.
 
\begin{corollary}
Assume that $G$ is weakly amenable discrete, compact or abelian. Let $\mu$ be a probability measure on $G$. 
The noncommutative Poisson boundary $\widetilde{\cl H}_{\mu} $ is 
*-isomorphic to the crossed product of $G\rtimes_{\ga}\cl H_{\mu}$. 
\end{corollary}
\begin{proof}
It follows from the definition of the von Neumann algebra structure on $\cl H _{\mu}$ 
that $\ga_s(\phi\diamond\psi)=\ga_s(\phi)\diamond\ga_s(\psi)$ for $\phi,\psi\in\cl H _{\mu}$.
Thus $G$ acts  
on the von Neumann algebra $\cl H_{\mu}$ by weak-* continuous *-automorphisms.
The Corollary now follows  from Proposition \ref{crosssystems}  and the fact that $\tilde\Gamma$ induces a 
completely positive surjective isometry  between von Neumann algebras, which must therefore by 
a *-homomorphism \cite[Corollary 5.2.3]{er}.
\end{proof}

\end{document}